\documentclass[leqno,12pt]{article}
\usepackage{amsfonts}

\usepackage{amsmath, amssymb}
\usepackage{amsthm}

\theoremstyle{plain} 
\newtheorem{theorem}{\indent\sc Theorem}[section]
\newtheorem{lemma}[theorem]{\indent\sc Lemma}

\newtheorem{proposition}[theorem]{\indent\sc Proposition}

\theoremstyle{definition} 
\newtheorem{definition}[theorem]{\indent\sc Definition}
\newtheorem{remark}[theorem]{\indent\sc Remark}
\newtheorem{example}[theorem]{\indent\sc Example}


\begin{document}

\title{On the Geometry of Lightlike Submanifolds in Metallic Semi-Riemannian
Manifolds}
\author{Feyza Esra Erdo\u{g}an, Selcen Y\"{u}ksel Perkta\c{s}, B\.{i}lal
Eftal Acet, \\
and Adara Monica Blaga}
\date{}
\maketitle

\begin{abstract}
In the present paper, we introduce screen transversal lightlike submanifolds
of metallic semi-Riemannian manifolds with its subclasses, namely screen
transversal anti-invariant, radical screen transversal and isotropic screen
transversal lightlike submanifolds, and give an example. We show that there
do not exist co-isotropic and totally screen transversal type of screen
transversal anti-invariant lightlike submanifolds of a metallic
semi-Riemannian manifold. We investigate the geometry of distributions
involved in the definition of such submanifolds and the conditions for the
induced connection to be a metric connection. Furthermore, we give a
necessary and sufficient condition for an isotropic screen transversal
lightlike submanifold to be totally geodesic.
\end{abstract}

\markboth{{\small\it {\hspace{8cm}On the Geometry of Lightlike Submanifolds}}}{\small\it{On the Geometry of Lightlike Submanifolds
\hspace{8cm}}}


\section{Introduction}

In Riemannian geometry, it is well known that the induced metric on a
submanifold of a Riemannian manifold is always a Riemannian one. But in
semi-Riemannian manifolds the induced metric by the semi-Riemannian metric
on the ambient manifold is not necessarily non-degenerate. This case leads
to provide an interesting type of submanifolds called lightlike
submanifolds. Because of degeneracy\ of the induced metric on lightlike
submanifolds, the tools which are used to investigate the geometry of
submanifolds in Riemannian case are not applicable in semi-Riemannian case
and so the classical theory fails while defining any induced object on a
lightlike submanifold. The main difficulties arise from the fact that the
intersection of the normal bundle and the tangent bundle of a lightlike
submanifold is nonzero. In order to resolve the difficulties that arise
during studying lightlike submanifolds, K. Duggal, A. Bejancu \cite{DB}
introduced a non-degenerate distribution called screen distribution to
construct a lightlike transversal vector bundle which does not intersect to
its lightlike tangent bundle. It is well-known that a suitable choice of
screen distribution gives rises to many substantial characterizations in
lightlike geometry. Many authors have studied the geometry of lightlike
submanifolds in different manifolds (see \cite%
{FC,BC,FBR,FBR1,SBE,BIS,BSE1,BIS1}). For a more comprehensive reading, we
refer \cite{DB,DS} and the references therein.

Different kinds of geometric structures (such as almost product, almost
contact, almost paracontact etc.) allow to get rich results while studying
on submanifolds. Recently, Riemannian manifolds with metallic structures are
one of the most studied topics in differential geometry.

In 2002, as a generalization of the Golden mean, metallic means family was
introduced by V. W. de Spinadel \cite{V.W4}, which contains the Silver mean,
the Bronze mean, the Copper mean and the Nickel mean etc. The positive
solution of the equation given by 
\begin{equation*}
x^{2}-px-q=0,
\end{equation*}%
for some positive integer $p$ and $q$, is called a $(p,q)$-metallic number 
\cite{V.W1,V.W3} and it has the form%
\begin{equation*}
\sigma _{p,q}=\frac{p+\sqrt{p^{2}+4q}}{2}.
\end{equation*}%
For $p=q=1$ and $p=2,$ $q=1$, it is well-known that we have the Golden mean $%
\phi =\frac{1+\sqrt{5}}{2}$ and Silver mean $\sigma _{2,1}=1+\sqrt{2}$,
respectively$.$ The metallic mean family plays an important role to
establish a relationship between mathematics and architecture. For example,
Golden mean and Silver mean can be seen in the sacred art of Egypt, Turkey,
India, China and other ancient civilizations \cite{V.W5}.

Polynomial structures on manifolds were introduced by S. I. Goldberg, K.
Yano and N. C. Petridis in (\cite{GY} and \cite{GY1}). C. E. Hretcanu and M.
Crasmareanu defined Golden structure as a particular case of polynomial
structure \cite{GH,GH2,GH3} and some generalizations of this, called
metallic structure \cite{GCS}. Being inspired by the metallic mean, the
notion of metallic manifold $\breve{N}$ was defined in \cite{CM} by a $(1,1)$%
-tensor field $\breve{J}$ on $\breve{N},$ which satisfies $\breve{J}^{2}=p%
\breve{J}+qI$, where $I$ is the identity operator on $\Gamma (T \breve{N})$
and $p$, $q$ are fixed positive integer numbers. Moreover, if $(\breve{N},%
\breve{g})$ is a Riemannian manifold endowed with a metallic structure $%
\breve{J}$ such that the Riemannian metric $\breve{g}$ is $\breve{J}$%
-compatible, i.e., $\breve{g}(\breve{J}V,W)=\breve{g}(V,\breve{J}W),$ for
any $V,W\in \Gamma (T\breve{N})$, then $(\breve{g},\breve{J})$ is called
metallic Riemannian structure and $(\breve{N},\breve{g},\breve{J})$ is a
metallic Riemannian manifold. Metallic structure on the ambient Riemannian
manifold provides important geometrical results on the submanifolds, since
it is an important tool while investigating the geometry of submanifolds.
Invariant, anti-invariant, semi-invariant, slant, semi-slant and hemi-slant
submanifolds of a metallic Riemannian manifold were studied in \cite%
{MH,MH1,MH2,MH3} and the authors obtained important characterizations on
submanifolds of metallic Riemannian manifolds. One of the most important
subclass of metallic Riemannian manifolds consists of the Golden Riemannian
manifolds. In recent years, many authors have studied Golden Riemannian
manifolds and their submanifolds (see \cite{BM,EC,M,FC3}). N. Poyraz \"{O}%
nen and E. Ya\c{s}ar \cite{NE} initiated the study of lightlike geometry in
Golden semi-Riemannian manifolds, by investigating lightlike hypersurfaces
of Golden semi-Riemannian manifolds. B. E. Acet introduced lightlike
hypersurfaces of a metallic semi-Riemannian manifold \cite{arxiv-eftal}.
Transversal lightlike submanifolds in metallic semi-Riemannian manifolds
were firstly studied by F. E. Erdo\u{g}an \cite{arxiv-feyza}.

In the present paper, we introduce screen transversal lightlike submanifolds
of metallic semi-Riemannian manifolds with its subclasses, namely screen
transversal anti-invariant, radical screen transversal and isotropic screen
transversal lightlike submanifolds, and give an example. We show that there
do not exist co-isotropic and totally screen transversal type of screen
transversal anti-invariant lightlike submanifolds of a metallic
semi-Riemannian manifold. We investigate the geometry of distributions
involved in the definition of such submanifolds and find necessary and
sufficient conditions for the induced connection to be a metric connection.
Furthermore, we give a necessary and sufficient condition for an isotropic
screen transversal lightlike submanifold to be totally geodesic.

\section{Preliminaries}

A submanifold $\acute{N}^{m}$ immersed in a semi-Riemannian manifold $(%
\breve{N}^{m+k},\breve{g})$ is called a lightlike submanifold if it admits a
degenerate metric $g$ induced from $\breve{g},$ whose radical distribution $%
Rad(T\acute{N})$ is of rank $r$, where $1\leq r\leq m.$ Then $Rad(T\acute{N}%
)=T\acute{N}\cap T\acute{N}^{\perp }$, where 
\begin{equation}
T\acute{N}^{\perp }=\cup _{x\in \acute{N}}\left\{ u\in T_{x}\breve{N}\mid 
\breve{g}\left( u,v\right) =0,\forall v\in T_{x}\acute{N}\right\} .
\end{equation}%
Let $S(T\acute{N})$ be a screen distribution which is a semi-Riemannian
complementary distribution of $Rad(T\acute{N})$ in $T\acute{N}$, i.e., $T%
\acute{N}=Rad(T\acute{N})\oplus_{ort} S(T\acute{N}).$

We consider a screen transversal vector bundle $S(T\acute{N}^{\bot }),$
which is a semi-Rie\-mannian complementary vector bundle of $Rad(T\acute{N})$
in $T\acute{N}^{\bot }.$ For any local basis $\left\{ \xi _{i}\right\} $ of $%
Rad(T\acute{N})$, there exists a lightlike transversal vector bundle $ltr(T%
\acute{N})$ locally spanned by $\left\{ N_{i}\right\} $ \cite{DB}. Let $tr(T%
\acute{N})$ be complementary (but not orthogonal) vector bundle to $T\acute{N%
}$ in $T\breve{N}^{\perp }|_{\acute{N}}$. Then we have%
\begin{eqnarray*}
tr(T\acute{N}) &=&ltr(T\acute{N})\,\oplus_{ort} \,S(T\acute{N}^{\bot }), \\
T\breve{N}\,|\,_{\acute{N}} &=&S(T\acute{N})\,\oplus_{ort} \,[Rad(T\acute{N}%
)\oplus ltr(T\acute{N})]\oplus_{ort} S(T\acute{N}^{\bot }).
\end{eqnarray*}%
Although $S(T\acute{N})$ is not unique, it is canonically isomorphic to the
factor vector bundle $T\acute{N}/Rad(T\acute{N})$ \cite{DB}.

Note that the lightlike second fundamental forms of a lightlike submanifold $%
\acute{N}$ do not depend on $S(T\acute{N}),$ $S(T\acute{N}^{\perp })$ and $%
ltr(T\acute{N})$ \cite{DB}.

We say that a submanifold $(\acute{N},g,S(T\acute{N}),$ $S(T\acute{N}^{\perp
}))$ of $\breve{N}$ is

Case 1: \ $r$-lightlike if $r<\min \{m,k\};$

Case 2: \ Co-isotropic if $r=k<m;S(T\acute{N}^{\perp })=\{0\};$

Case 3: \ Isotropic if $r=m<k;$ $S(T\acute{N})=\{0\};$

Case 4: \ Totally lightlike if $r=k=m;$ $S(T\acute{N})=\{0\}=S(T\acute{N}%
^{\perp }).$

The Gauss and Weingarten equations are given by%
\begin{eqnarray}
\breve{\nabla}_{W}U &=&\nabla _{W}U+h\left( W,U\right),\quad \forall W,U\in
\Gamma (T\acute{N}),  \label{7} \\
\breve{\nabla}_{W}V &=&-A_{V}W+\nabla _{W}^{t}V,\quad \forall W\in \Gamma (T%
\acute{N}),V\in \Gamma (tr(T\acute{N})),  \label{8}
\end{eqnarray}%
where $\left\{ \nabla _{W}U,A_{V}W\right\} $ and $\left\{ h\left( W,U\right)
,\nabla _{W}^{t}V\right\} $ belong to $\Gamma (T\acute{N})$ and $\Gamma (tr(T%
\acute{N})),$ respectively. Here, $\nabla $ and $\nabla ^{t}$ denote linear
connections on $\acute{N}$ and the vector bundle $tr\,(T\acute{N})$,
respectively. Also, for any $W,U\in $ \text{$\Gamma (T\acute{N})$}, $N\in $ $%
\Gamma (ltr(T\acute{N}))$ and $Z\in $ $\Gamma (S(T\acute{N}^{\bot }))$, we
have%
\begin{eqnarray}
\breve{\nabla}_{W}U &=&\nabla _{W}U+h^{\ell }\left( {\ W,U}\right)
+h^{s}\left( {\ W,U}\right),  \label{9} \\
\breve{\nabla}_{W}N &=&-A_{N}W+\nabla _{W}^{\ell }N+D^{s}\left( {\ W,N}%
\right),  \label{10} \\
\breve{\nabla}_{W}Z &=&-A_{Z}W+\nabla _{W}^{s}Z+D^{\ell }\left( {\ W,Z}%
\right).  \label{11}
\end{eqnarray}

Let $P$ denotes the projection of $T\acute{N}$ on $S(T\acute{N}).$ Since $%
\breve{\nabla}$ is a metric connection, then by using (\ref{7}), (\ref{9})-(%
\ref{11}) we get%
\begin{eqnarray}
\breve{g}(h^{s}\left( {\ W,U}\right) ,Z)+\breve{g}({\ U,D}^{\ell }\left( {\
W,Z}\right) ) &=&\breve{g}\left( {\ A}_{Z}{\ W,U}\right) ,  \label{12} \\
\breve{g}\left( {\ D}^{s}\left( {\ W,N}\right) {\ ,Z}\right) &=&\breve{g}%
\left( {\ N,A}_{Z}{\ W}\right) .  \label{13}
\end{eqnarray}%
From the decomposition of the tangent bundle of a lightlike submanifold, we
write%
\begin{eqnarray}
\nabla _{W}PU &=&\nabla _{W}^{\ast }PU+h^{\ast }\left( W,PU\right) ,
\label{14} \\
\nabla _{W}\xi &=&-A_{\xi }^{\ast }W+\nabla _{W}^{\ast t}\xi ,  \label{15}
\end{eqnarray}%
for $W,U\in \Gamma (T\acute{N})$ and $\xi \in \Gamma (Rad(T\acute{N})),$
which imply%
\begin{eqnarray}
g\left( h^{\ell }\left( W,PU\right) ,\xi \right) &=&g\left( A_{\xi }^{\ast
}W,PU\right) ,  \label{16} \\
g\left( h^{s}\left( W,PU\right) ,N\right) &=&g\left( A_{N}W,PU\right) ,
\label{17} \\
g\left( h^{\ell }\left( W,\xi \right) ,\xi \right) &=&0,\quad A_{\xi }^{\ast
}\xi =0.  \label{18}
\end{eqnarray}%
In general, the induced connection $\nabla $ on $\acute{N}$ is not a metric
connection. Since $\breve{\nabla}$ is a metric connection, by using (\ref{9}%
) we get%
\begin{equation}
\left( \nabla _{W}g\right) \left( U,V\right) {\ =g}\left( h^{\ell }\left(
W,U\right) ,V\right) {\ +g}\left( h^{\ell }\left( W,V\right) ,U\right) .
\label{19}
\end{equation}%
However, we note that $\nabla ^{\ast }$ is a metric connection on $S(T\acute{%
N})$.

The positive solution of the equation 
\begin{equation*}
x^{2}-px-q=0,
\end{equation*}%
for fixed two positive integers $p$ and $q,$ is called a member of metallic
means family (\cite{V.W1}-\cite{V.W5}). These numbers, given by%
\begin{equation}
\sigma _{p,q}=\frac{p+\sqrt{p^{2}+4q}}{2},  \label{19a}
\end{equation}%
are called $(p,q)$-metallic numbers.

\begin{definition}
\cite{CM} A polynomial structure on a manifold $\breve{N}$ is called a
metallic structure if it is determined by an $(1,1)$-tensor field $\breve{J} 
$ which satisfies \ 
\begin{equation}
\breve{J}^{2}=p\breve{J}+qI,  \label{20}
\end{equation}%
where $I$ is the identity map on $\Gamma (T\breve{N})$ and $p,q$ are
positive integers. Also, if 
\begin{equation}
\breve{g}(\breve{J}W,U)=\breve{g}(W,\breve{J}U)  \label{21}
\end{equation}%
\newline
holds for every $U,W\in \Gamma (T\breve{N})$, then the semi-Riemannian
metric $\breve{g}$ is called $\breve{J}$-compatible. In this case, $(\breve{N%
},\breve{g},\breve{J})$ is called a metallic semi-Riemannian manifold.
Furthermore, a metallic semi-Riemannian structure $\breve{J}$ is called a
locally metallic structure if $\breve{J}$ is parallel with respect to the
Levi-Civita connection $\breve{\nabla},$ that is 
\begin{equation}
\breve{\nabla}_{W}\breve{J}U=\breve{J}\breve{\nabla}_{W}U.  \label{21a}
\end{equation}
\end{definition}

If $\breve{J}$ is a metallic structure, then (\ref{21}) is equivalent to 
\cite{CM} 
\begin{equation}
\breve{g}(\breve{J}W,\breve{J}U)=p\breve{g}(\breve{J}W,U)+q\breve{g}(W,U),
\label{23}
\end{equation}%
for any $W,U\in \Gamma (T\breve{N}).$

It is known, from \cite{GY1}, that a polynomial structure on a manifold $%
\breve{N}$ defined by a smooth tensor field of type $(1,1)$ induces a
generalized almost product structure $F$, i.e., $F^{2}=I$, on $\breve{N}$
with number of distributions of $F$ equal to the number of distinct
irreducible factors of the structure polynomial over the real field while
the projectors are expressed as polynomials in $F$.

\begin{proposition}
\cite{CM} Every almost product structure $F$ induces two metallic structures
on $\breve{N}$ given as follows:%
\begin{equation*}
\breve{J}_{1}=\frac{p}{2}I+\left( \frac{2\sigma _{p,q}-p}{2}\right) F,\text{
\ \ \ \ \ }\breve{J}_{2}=\frac{p}{2}I-\left( \frac{2\sigma _{p,q}-p}{2}%
\right) F.
\end{equation*}%
Conversely, every metallic structure $\breve{J}$ on $\breve{N}$ induces two
almost product structures on this manifold:%
\begin{equation*}
F=\pm \left( \frac{2}{2\sigma _{p,q}-p}\breve{J}-\frac{p}{2\sigma _{p,q}-p}%
I\right) .
\end{equation*}
\end{proposition}

\section{Screen Transversal Lightlike Submanifolds of Metallic
Semi-Riemannian Manifolds}

In this section, before introducing a screen transversal lightlike
submanifold of metallic semi-Riemannian manifolds, we begin with the
following.

\begin{lemma}
Let $\acute{N}$ be a lightlike submanifold of a metallic semi-Riemannian
manifold $(\breve{N},\breve{g},\breve{J})$ with a vector subbundle $ltr(T%
\acute{N})$ of the screen transversal vector bundle. Then we have%
\begin{equation*}
\breve{J}Rad(T\acute{N})\cap \breve{J}ltr(T\acute{N})=\{0\}.
\end{equation*}
\end{lemma}

\begin{proof}
Assume that $ltr(T\acute{N})$ is invariant with respect to $\breve{J},$ i.e., $%
\breve{J}ltr(T\acute{N})=ltr(T\acute{N}).$ From the definition of a lightlike
submanifold, we have
\begin{equation}
\breve{g}(N,\xi )=1,  \label{a}
\end{equation}%
for $\xi \in \Gamma (Rad(T\acute{N}))$ and $N\in \Gamma (ltr(T\acute{N})).$ Also
from (\ref{23}), we find that%
\begin{equation*}
\breve{g}(\breve{J}N,\breve{J}\xi )=p+q.
\end{equation*}%
However, since $\breve{J}N\in \Gamma (ltr(T\acute{N})),$ then by hypothesis,
we get $\breve{g}(\breve{J}N,\breve{J}\xi )=0,$ which is a contradiction. So
$\breve{J}N$ can not belong to $\Gamma (ltr(T\acute{N}))$.

Now, let us consider $\breve{J}N\in \Gamma (S(T\acute{N})).$ Then we obtain
$\breve{g}(\breve{J}N,\breve{J}\xi )=0$, which contradicts (\ref{a}).
When we assume $\breve{J}N\in \Gamma (Rad(T\acute{N}))$,
we get the same contradiction. Thus, $\breve{J}N$ does not belong to $S(T\acute{N})$ as well as
to $Rad(T\acute{N}).$ Then, from the decomposition of a lightlike submanifold, we
conclude that $\breve{J}N\in \Gamma (S(T\acute{N}^{\perp })).$ This
completes the proof.
\end{proof}

\begin{definition}
\label{def1}Let $\acute{N}$ be a lightlike submanifold of a metallic
semi-Riemannian manifold $(\breve{N},\breve{g},\breve{J}).$ If 
\begin{equation*}
\breve{J}Rad(T\acute{N})\subset S(T\acute{N}^{\perp }),
\end{equation*}%
then $\acute{N}$ is called a screen transversal lightlike submanifold of a
metallic semi-Riemannian manifold.
\end{definition}

\begin{example}
Let $\left( \breve{N}=\mathbb{R}_{2}^{5},\breve{g},\breve{J}\right) $ be the
five-dimensional semi-Euclidean space with the semi-Euclidean metric $\breve{%
g}$ of sign $(-,-,+,+,+).$ If we take 
\begin{equation*}
\breve{J}(x_{1},x_{2},x_{3},x_{4},x_{5})=(\left( p-\sigma \right)
x_{1},\left( p-\sigma \right) x_{2},\sigma x_{3},\sigma x_{4},\sigma x_{5}),
\end{equation*}
where $(x_{1},x_{2},x_{3},x_{4},x_{5})$ is the standard coordinate system of 
$\mathbb{R}_{2}^{5}$, then one can easily see that $\breve{J}$ is a metallic
structure on $\mathbb{R}_{2}^{5}.$ Let $\acute{N}$ be a submanifold in $%
\breve{N}$ defined by 
\begin{equation*}
x_{3}=0,\quad x_{5}=x_{1}+x_{2}.
\end{equation*}%
Then we get $T\acute{N}=Sp\{W_{1},W_{2},W_{3}\},$ for%
\begin{equation*}
W_{1}=\frac{\partial }{\partial x_{1}}+\frac{\partial }{\partial x_{5}},%
\text{ \ }W_{2}=\frac{\partial }{\partial x_{2}}+\frac{\partial }{\partial
x_{5}},\text{ \ }W_{3}=\frac{\partial }{\partial x_{4}}.
\end{equation*}%
It is easy to check that $\acute{N}$ is a lightlike submanifold. Therefore,%
\begin{eqnarray*}
Rad(T\acute{N}) &=&Sp\{W_{1}=\xi \}, \\
ltr(T\acute{N}) &=&Sp\left\{ N=\frac{1}{2}\left( -\frac{\partial }{\partial
x_{1}}+\frac{\partial }{\partial x_{5}}\right) \right\} , \\
S(T\acute{N}) &=&Sp\{W_{4}\},
\end{eqnarray*}%
and we have%
\begin{eqnarray*}
\breve{J}\xi &=&\left( p-\sigma \right) \frac{\partial }{\partial x_{1}}%
+\sigma \frac{\partial }{\partial x_{5}}\in S(T\acute{N}^{\perp }), \\
\breve{J}N &=&\frac{1}{2}\left( -\left( p-\sigma \right) \frac{\partial }{%
\partial x_{1}}+\sigma \frac{\partial }{\partial x_{5}}\right) \in S(T\acute{%
N}^{\perp }).
\end{eqnarray*}%
Thus, $\acute{N}$ is a radical screen transversal lightlike submanifold of $%
\breve{N}.$
\end{example}

\begin{definition}
Let $\acute{N}$ be a screen transversal lightlike submanifold of a metallic
semi-Riemannian manifold $(\breve{N},\breve{g},\breve{J}).$
\end{definition}

\begin{enumerate}
\item If $\breve{J}S(T\acute{N})\subset S(T\acute{N}^{\perp }),$ then we say
that $\acute{N}$ is a screen transversal anti-invariant submanifold of $(%
\breve{N},\breve{g},\breve{J}).$

\item If $\breve{J}S(T\acute{N})=S(T\acute{N}),$ then we say that $\acute{N}$
is a radical screen transversal lightlike submanifold of $(\breve{N},\breve{g%
},\breve{J}).$
\end{enumerate}

\begin{remark}
Let $\acute{N}$ be a lightlike submanifold of a metallic semi-Riemannian
manifold $(\breve{N},\breve{g},\breve{J}).$ Considering the definition of a
lightlike submanifold, we note the followings \cite{DS}:
\end{remark}

\begin{description}
\item[(i)] the radical distribution $Rad(T\acute{N})$ is integrable (resp.,
defines a totally geodesic foliation) if and only if 
\begin{equation}
g\left( \left[ W,U\right] ,Z\right) =0,\quad (resp.,\quad \breve{g}\left(
\nabla _{W}U,Z\right) =0),  \label{int-geo-1}
\end{equation}%
for $W,U\in \Gamma (Rad(T\acute{N}))$ and $Z\in \Gamma (S(T\acute{N})).$

\item[(ii)] the screen distribution $S(T\acute{N})$ is integrable (resp.,
defines a totally geodesic foliation) if and only if 
\begin{equation}
g\left( \left[ W,U\right] ,N\right) =0,\quad (resp.,\quad \breve{g}\left(
\nabla _{W}U,N\right) =0),  \label{int-geo-2}
\end{equation}%
for $W,U\in \Gamma (S(T\acute{N}))$ and $N\in \Gamma (ltr(T\acute{N})).$
\end{description}

\subsection{Screen Transversal Anti-Invariant Submanifolds}

Let $\acute{N}$ be a screen transversal anti-invariant submanifold of $(%
\breve{N},\breve{g},\breve{J})$. Then we have following decomposition: 
\begin{equation*}
S(T\acute{N}^{\perp })=\{\breve{J}Rad(T\acute{N})\oplus \breve{J}ltr(T\acute{%
N})\oplus \breve{J}S(T\acute{N})\}\oplus_{ort} D_{o},
\end{equation*}%
where, $D_{o}$ is a non-degenerate distribution orthogonal complement to $%
\breve{J}Rad(T\acute{N})\oplus \breve{J}ltr(T\acute{N})\oplus \breve{J}S(T%
\acute{N})$.

\begin{proposition}
\label{prop2}Let $\acute{N}$ be a screen transversal anti-invariant
lightlike submanifold of a metallic semi-Riemannian manifold $(\breve{N},%
\breve{g},\breve{J})$. Then the distribution $D_{o}$ is invariant with
respect to $\breve{J}.$
\end{proposition}

\begin{proof}
Using (\ref{21}), we obtain%
\begin{equation*}
\breve{g}(\breve{J}W,\xi )=\breve{g}(W,\breve{J}\xi )=0,
\end{equation*}%
which implies that $\breve{J}U$ does not belong to $\Gamma(ltr(T\acute{N})).$ Since we
have%
\begin{eqnarray*}
\breve{g}(\breve{J}W,N) &=&\breve{g}(W,\breve{J}N)=0, \\
\breve{g}(\breve{J}W,\breve{J}\xi ) &=&\breve{g}(W,\breve{J}\xi )+\breve{g}%
(W,\xi )=0, \\
\breve{g}(\breve{J}W,\breve{J}N) &=&0, \\
\breve{g}(\breve{J}W,U) &=&\breve{g}(W,\breve{J}U)=0, \\
\breve{g}(\breve{J}W,\breve{J}U) &=&0
\end{eqnarray*}%
for $W\in \Gamma \left( D_{o}\right) ,$ $\xi \in \Gamma (Rad(T\acute{N})),$ $%
N\in \Gamma (ltr(T\acute{N}))$ and $U\in \Gamma (S(T\acute{N})),$ then we
complete the proof.
\end{proof}

\begin{proposition}
Let $\acute{N}$ be a screen transversal anti-invariant lightlike submanifold
of a metallic semi-Riemannian manifold $(\breve{N},\breve{g},\breve{J})$.
Then there do not exist co-isotropic and totally screen transversal type of
such lightlike submanifolds.
\end{proposition}

\begin{proof}
If $\acute{N}$ is a co-isotropic or totally screen transversal lightlike
submanifold, then we have
\begin{equation*}
S(T\acute{N}^{\perp })=\{0\}.
\end{equation*}%
Therefore, from Definition \ref{def1}, the proof is trivial.
\end{proof}

Assume that $\breve{N}$ is a screen transversal anti-invariant submanifold
of a metallic semi-Riemannian manifold $(\breve{N},\breve{g},\breve{J})$.
Let $\ T_{1},T_{2},T_{3},$ and $T_{4\text{ }}$ be the projection morphisms
on $\breve{J}Rad(T\acute{N}),$ $\breve{J}S(T\acute{N}),$ $\breve{J}ltr(T%
\acute{N}),$ and $D_{o},$ respectively. Then, for $U\in \Gamma (S(T\acute{N}%
^{\perp })),$ we have expression 
\begin{equation}
U=T_{1}U+T_{2}U+T_{3}U+T_{4\text{ }}U.  \label{26}
\end{equation}%
If we apply $\breve{J}$ to (\ref{26}), then we find 
\begin{equation}
\breve{J}U=\breve{J}T_{1}U+\breve{J}T_{2}U+\breve{J}T_{3}U+\breve{J}T_{4%
\text{ }}U.  \label{27}
\end{equation}%
On the other hand, we have%
\begin{equation}
\breve{J}U=BU+CU,  \label{26a}
\end{equation}%
for $U\in \Gamma (S(T\acute{N}^{\perp })),$ where, $BU$ and $CU$ are the
tangent and transversal components of $\breve{J}U$, respectively.

Also, let $R$ and $R^{\prime }$ be the projection morphisms of $\breve{J}%
T_{1}U$ on $Rad(T\acute{N})$ and $\breve{J}Rad(T\acute{N})$, respectively$;$ 
$S$ and $S^{\prime }$ be the projection morphisms of $\breve{J}T_{2}U$ on $%
S(T\acute{N})$ and $\breve{J}S(T\acute{N})$, respectively; $L$ and $%
L^{\prime }$ be the projection morphisms of $\breve{J}T_{3}U$ on $ltr(T%
\acute{N})$ and $\breve{J}ltr(T\acute{N})$, respectively$;$ $D$ and $%
D^{\prime }$ be the projection morphisms of $\breve{J}T_{4\text{ }}U$ on $%
D_{o}$ and $\breve{J}D_{o}$, respectively$.$ Then, from (\ref{27}) and (\ref%
{26a}), we have 
\begin{eqnarray*}
BU &=&R\breve{J}T_{1}U+S\breve{J}T_{2}U, \\
CU &=&R^{\prime }\breve{J}T_{1}U+S^{\prime }\breve{J}T_{2}U+L\breve{J}T_{3}U
\\
&&+L^{\prime }\breve{J}T_{3}U+D\breve{J}T_{4\text{ }}U+D^{\prime }\breve{J}%
T_{4\text{ }}U.
\end{eqnarray*}%
If we put $B_{1}=R\breve{J}T_{1},$ $B_{2}=S\breve{J}T_{2},$ $C_{1}=L\breve{J}%
T_{3},$ and $C_{2}=$ $R^{\prime }\breve{J}T_{1}+S^{\prime }\breve{J}T_{2}+L%
\breve{J}T_{3}+L^{\prime }\breve{J}T_{3}+D\breve{J}T_{4\text{ }}+D^{\prime }%
\breve{J}T_{4\text{ }},$ then we can rewrite (\ref{27}) as%
\begin{equation}
\breve{J}U=B_{1}U+B_{2}U+C_{1}U+C_{2}U.  \label{28}
\end{equation}%
Here there are components of $B_{1}U$, $B_{2}U$, $C_{1}U$, and $C_{2}U$ at$\
\Gamma (Rad(T\acute{N}))$, $\Gamma (S(T\acute{N}))$, $\Gamma (ltr(T\acute{N}%
))$, and $\Gamma (S(T\acute{N}^{\perp }))$, respectively, namely $\breve{J}U$
belongs to $T\breve{N}\,|_{\acute{N}}$.

It is known that the induced connection on a screen transversal
anti-invariant lightlike submanifold immersed in metallic semi-Riemannian
manifolds is not a metric connection. The condition under which the induced
connection on the submanifold would be a metric connection is given by the
following theorem.

\begin{theorem}
Let $\acute{N}$ be a screen transversal anti-invariant lightlike submanifold
of a locally metallic semi-Riemannian manifold $(\breve{N},\breve{g},\breve{J%
})$. Then the induced connection $\nabla $ on $\acute{N}$ is a metric
connection if and only if 
\begin{equation*}
B_{2}\nabla _{W}^{s}\breve{J}\xi =0,
\end{equation*}%
for $W\in \Gamma (T\acute{N})$ and $\xi \in \Gamma (Rad(T\acute{N})).$
\end{theorem}

\begin{proof}
From (\ref{21a}), for $W\in \Gamma (T\acute{N})$ and $\xi \in \Gamma (Rad(T%
\acute{N})),$ we have%
\begin{equation*}
\breve{\nabla}_{W}\breve{J}\xi =\breve{J}\breve{\nabla}_{W}\xi .
\end{equation*}%
If we use (\ref{9}) and (\ref{11}), then we get%
\begin{equation*}
-A_{\breve{J}\xi }W+\nabla _{W}^{s}\breve{J}\xi +D^{l}(W,\breve{J}\xi )=%
\breve{J}\left( \nabla _{W}\xi +h^{l}(W,\xi \right) +h^{s}(W,\xi )).
\end{equation*}%
Applying $\breve{J}$ to above equation, we find
\begin{equation*}
-\breve{J}A_{\breve{J}\xi }W+\breve{J}\nabla _{W}^{s}\breve{J}\xi +\breve{J}%
D^{l}(W,\breve{J}\xi ){ =\breve{J}}^{2}(\nabla _{W}\xi +h^{l}(W,\xi )%
{ +h}^{s}{ (W,\xi ))}.
\end{equation*}%
Then from (\ref{20}), we obtain%
\begin{equation*}
-\breve{J}A_{\breve{J}\xi }W+\breve{J}\nabla _{W}^{s}\breve{J}\xi +\breve{J}%
D^{l}(W,\breve{J}\xi )=\left(
\begin{array}{c}
p\breve{J}\nabla _{W}\xi +p\breve{J}h^{l}(W,\xi ) \\
+p\breve{J}{ h}^{s}{ (W,\xi )}+q\nabla _{W}\xi  \\
+qh^{l}(W,\xi )+q{ h}^{s}{ (W,\xi )}%
\end{array}%
\right) .
\end{equation*}%
If we use (\ref{28}) in last equation above, we can write%
\begin{equation*}
\left(
\begin{array}{c}
-\breve{J}A_{\breve{J}\xi }W+B_{1}\nabla _{W}^{s}\breve{J}\xi  \\
+B_{2}\nabla _{W}^{s}\breve{J}\xi +C_{1}\nabla _{W}^{s}\breve{J}\xi  \\
+C_{2}\nabla _{W}^{s}\breve{J}\xi +\breve{J}D^{l}\left( W,\breve{J}\xi
\right)
\end{array}%
\right) =\left(
\begin{array}{c}
p\breve{J}\nabla _{W}\xi +p\breve{J}h^{l}(W,\xi ) \\
+p\breve{J}{ h}^{s}{ (W,\xi )}+q\nabla _{W}\xi  \\
+qh^{l}(W,\xi )+q{ h}^{s}{ (W,\xi )}%
\end{array}%
\right) .
\end{equation*}%
By equating the tangent parts of the last equation, we have
\begin{equation*}
\frac{1}{q}(B_{1}\nabla _{W}^{s}\breve{J}\xi +B_{2}\nabla _{W}^{s}\breve{J}%
\xi )=\nabla _{W}\xi .
\end{equation*}%
Hence, the proof is completed.
\end{proof}

\begin{theorem}
Let $\acute{N}$ be a screen transversal anti-invariant lightlike submanifold
of a locally metallic semi-Riemannian manifold $(\breve{N},\breve{g},\breve{J%
})$. Then the radical distribution is integrable if and only if 
\begin{equation*}
\nabla _{W}^{s}\breve{J}V=\nabla _{V}^{s}\breve{J}W,
\end{equation*}%
for $V,W\in \Gamma (Rad(T\acute{N})).$
\end{theorem}

\begin{proof}
From (\ref{int-geo-1}), we get
\begin{eqnarray*}
{ 0}{ =g} &&(\breve{\nabla}_{W}\breve{J}V,\breve{J}Z){ -pg(%
}\breve{\nabla}_{W}V,\breve{J}Z) \\
&&{ -g(}\breve{\nabla}_{V}\breve{J}W,\breve{J}Z){ +pg(}\breve{%
\nabla}_{V}W,\breve{J}Z),
\end{eqnarray*}%
for $V,W\in \Gamma (Rad(T\acute{N}))$ and $Z\in \Gamma (S(T\acute{N})).$ Since
$\breve{J}U,\breve{J}W\in \Gamma (S(T\acute{N}^{\perp })),$ then by using (%
\ref{10}), we find%
\begin{equation*}
0=g(\nabla _{W}^{s}\breve{J}V-\nabla _{V}^{s}\breve{J}W,\breve{J}Z).
\end{equation*}%
\end{proof}

\begin{theorem}
Let $\acute{N}$ be a screen transversal anti-invariant lightlike submanifold
of a locally metallic semi-Riemannian manifold $(\breve{N},\breve{g},\breve{J%
})$. In this case, the screen distribution\ is integrable if and only if 
\begin{equation*}
\nabla _{W}^{s}\breve{J}U=\nabla _{U}^{s}\breve{J}W,
\end{equation*}%
for $W,U\in \Gamma (S(T\acute{N})).$
\end{theorem}

\begin{proof}
By using (\ref{int-geo-2}), from (\ref{21a}), (\ref{21}) and (\ref{23}), we
find%
\begin{eqnarray*}
{ 0} &{ =}&{ g(}\breve{\nabla}_{W}\breve{J}U,\breve{J}N)%
{ -pg(}\breve{\nabla}_{W}U,\breve{J}N) \\
&&{ -g(}\breve{\nabla}_{U}\breve{J}W,\breve{J}N){ +pg(}\breve{%
\nabla}_{U}W,\breve{J}N) \\
&{ =}&{ g(}\nabla _{W}^{s}\breve{J}U,\breve{J}N){ -g(}%
\nabla _{U}^{s}\breve{J}W,\breve{J}N) \\
&&{ -g(}ph^{s}(W,U),\breve{J}N){ +g(}ph^{s}(U,W),\breve{J}N)%
{ ,}
\end{eqnarray*}%
for $W,U\in \Gamma (S(T\acute{N}))$ and $N\in \Gamma (ltr(T\acute{N})).$ The
last equation implies%
\begin{equation*}
\nabla _{W}^{s}\breve{J}U-\nabla _{U}^{s}\breve{J}W=p\left(
h^{s}(W,U)-h^{s}(U,W)\right) .
\end{equation*}%
Since $h^{s}$ is symmetric, we get
$
\nabla _{W}^{s}\breve{J}U=\nabla _{U}^{s}\breve{J}W.
$
\end{proof}

\begin{theorem}
Let $\acute{N}$ be a radical screen transversal lightlike submanifold of a
locally metallic semi-Riemannian manifold $(\breve{N},\breve{g},\breve{J})$.
Then the radical distribution defines a totally geodesic foliation if and
only if there is no component of $\nabla _{W}^{s}\breve{J}U-ph^{s}(W,U)$ in $%
S(T\acute{N}),$ for $W,U\in \Gamma (Rad(T\acute{N})).$
\end{theorem}

\begin{proof}
From the second part of (\ref{int-geo-1}), for $W,U\in \Gamma (Rad(T\acute{N})%
)$ and $Z\in S(T\acute{N}),$ if we use (\ref{9}), (\ref{21a}) and (\ref{23}%
), we have
\begin{equation*}
\breve{g}(\breve{\nabla}_{W}\breve{J}U,\breve{J}Z)-p\breve{g}(\breve{\nabla}%
_{W}U,\breve{J}Z)=0.
\end{equation*}%
Then we find
\begin{equation*}
\breve{g}(\nabla _{W}^{s}\breve{J}U-ph^{s}(W,U),\breve{J}Z)=0,
\end{equation*}%
by virtue of (\ref{11}). So, the proof is completed.
\end{proof}

\begin{theorem}
Let $\acute{N}$ be a screen transversal lightlike submanifold of a locally
metallic semi-Riemannian manifold $(\breve{N},\breve{g},\breve{J})$. Then
the screen distribution defines a totally geodesic foliation if and only if
there is no component of $\ \nabla _{W}^{s}\breve{J}U-ph^{s}(W,U)$ in $%
\breve{J}ltr(T\acute{N}),$ for $W,U\in \Gamma (S(T\acute{N})).$
\end{theorem}

\begin{proof}
By using (\ref{int-geo-2}), (\ref{9}), (\ref{23}) and (\ref{21a}), we get%
\begin{equation*}
\breve{g}(\breve{\nabla}_{W}\breve{J}U,\breve{J}N)-p\breve{g}(\breve{\nabla}%
_{W}U,\breve{J}N)=0,
\end{equation*}%
for any $W,U\in \Gamma (S(T\acute{N})),$ $N\in \Gamma (ltr(T\acute{N})).$
Since $\breve{J}U\in \Gamma (S(T\acute{N}^{\perp })),$ from the (\ref{11}),
we obtain
\begin{equation*}
0=\breve{g}(\nabla _{W}^{s}\breve{J}U-ph^{s}(W,U),\breve{J}N).
\end{equation*}%
\end{proof}

\subsection{Radical Screen Transversal Lightlike Submanifolds of Metallic
Semi-Riemannian Manifolds}

We begin with investigating the integrability conditions of the
distributions.

\begin{theorem}
Let $\acute{N}$ be a radical screen transversal lightlike submanifold of a
locally metallic semi-Riemannian manifold $(\breve{N},\breve{g},\breve{J})$.
In this case, the screen distribution is integrable if and only if there is
no component of $h^{s}(W,\breve{J}U)-h^{s}(U,\breve{J}W)$ in $\breve{J}ltr(T%
\acute{N}),$ for any $W,U\in \Gamma (S(T\acute{N})).$
\end{theorem}

\begin{proof}
From (\ref{int-geo-2}), and then using (\ref{9}), (\ref{21}), (\ref{21a}), (%
\ref{23}), we find%
\begin{eqnarray*}
{ 0} &{ =}&{ g(}\breve{\nabla}_{W}\breve{J}U,\breve{J}N)%
{ -pg(}\breve{\nabla}_{W}U,\breve{J}N){ -g(}\breve{\nabla}_{U}%
\breve{J}W,\breve{J}N){ +pg(}\breve{\nabla}_{U}W,\breve{J}N) \\
&{ =}&{ g(}h^{s}(W,\breve{J}U)-h^{s}(U,\breve{J}%
W)-ph^{s}(W,U)+ph^{s}(U,W),\breve{J}N),
\end{eqnarray*}%
for $W,U\in \Gamma (S(T\acute{N}))$ and $N\in \Gamma (ltr(T\acute{N})).$ Here,
since $h^{s}$ is symmetric, then we have%
\begin{equation*}
\breve{g}(h^{s}(W,\breve{J}U)-h^{s}(U,\breve{J}W),\breve{J}N)=0.
\end{equation*}%
Therefore, the proof is completed.
\end{proof}

\begin{theorem}
Let $\acute{N}$ be a radical screen transversal lightlike submanifold of a
locally metallic semi-Riemannian manifold $(\breve{N},\breve{g},\breve{J})$.
Then \ the radical distribution is integrable if and only if either $A_{%
\breve{J}W}U-A_{\breve{J}U}W=p\left( A_{W}^{\ast }U-A_{U}^{\ast }W\right) $
or $A_{W}^{\ast }U=A_{U}^{\ast }W$ and $A_{\breve{J}W}U-A_{\breve{J}U}W$
belong to $\Gamma (Rad(T\acute{N})),$ for $W,U\in \Gamma (Rad(T\acute{N})).$
\end{theorem}

\begin{proof}
From (\ref{int-geo-1}), (\ref{21}), (\ref{21a}) and (\ref{23}), we have%
\begin{equation*}
{ 0=g(}\breve{\nabla}_{W}\breve{J}U,\breve{J}Z){ -g(}\breve{%
\nabla}_{U}\breve{J}W,\breve{J}Z){ -pg(}\breve{\nabla}_{W}U,\breve{J}Z)%
{ +pg(}\breve{\nabla}_{U}W,\breve{J}Z),
\end{equation*}%
for $W,U\in \Gamma (Rad(T\acute{N}))$ and $Z\in \Gamma (S(T\acute{N})).$ Since
$\breve{J}U,\,\breve{J}W\in \Gamma (S(T\acute{N}^{\perp }))$ and $\breve{J}%
Z\in \Gamma (S(T\acute{N}))$, from (\ref{9}) and (\ref{11}), we obtain%
\begin{equation*}
0=g(A_{\breve{J}U}W-A_{\breve{J}W}U-pA_{W}^{\ast }U+pA_{U}^{\ast }W,\breve{J}%
Z),
\end{equation*}%
which completes the proof.
\end{proof}

\begin{proposition}
Let $\acute{N}$ be a radical screen transversal lightlike submanifold of a
locally metallic semi-Riemannian manifold $(\breve{N},\breve{g},\breve{J})$.
Then the distribution $D_{o}$ is invariant with respect to $\breve{J}.$
\end{proposition}

\begin{proof}
For a radical screen transversal lightlike submanifold, we have
\begin{eqnarray*}
S(T\acute{N}^{\perp }) &=&\breve{J}Rad(T\acute{N})\oplus \breve{J}ltr(T\acute{N})%
\oplus_{ort} D_{o}, \\
\breve{J}S(T\acute{N}) &=&S(T\acute{N}).
\end{eqnarray*}%
Here, for $W\in D_{o}$ and $U\in \Gamma (S(T\acute{N})),$ if we use (\ref{21}%
) and (\ref{23}), we find
\begin{eqnarray*}
g(\breve{J}W,\xi ) &=&g(W,\breve{J}\xi )=0, \\
g(\breve{J}W,\breve{J}\xi ) &=&0, \\
g(\breve{J}W,N) &=&g(W,\breve{J}N)=0, \\
g(\breve{J}W,\breve{J}N) &=&0, \\
g(\breve{J}W,U) &=&g(W,\breve{J}U)=0, \\
g(\breve{J}W,\breve{J}U) &=&0.
\end{eqnarray*}%
Therefore, we obtain
\begin{eqnarray*}
\breve{J}D_{o}\cap \breve{J}Rad(T\acute{N}) &=&\{0\},\quad \breve{J}D_{o}\cap
\breve{J}ltr(T\acute{N})=\{0\}, \\
\breve{J}D_{o}\cap Rad(T\acute{N}) &=&\{0\},\quad \breve{J}D_{o}\cap ltr(T%
\acute{N})=\{0\}, \\
\breve{J}D_{o}\cap \breve{J}S(T\acute{N}) &=&\{0\},\quad \breve{J}D_{o}\cap
S(T\acute{N})=\{0\},
\end{eqnarray*}%
which give the assertion of the theorem.
\end{proof}

\begin{theorem}
Let $\acute{N}$ be a\ radical screen transversal lightlike submanifold of a
locally metallic semi-Riemannian manifold $(\breve{N},\breve{g},\breve{J})$.
Then the screen distribution defines a totally geodesic foliation if and
only if there is no component of $h^{s}(W,\breve{J}U)-ph^{s}(W,U)$ in $%
\breve{J}ltr(T\acute{N}),$ for any $W,U\in \Gamma (S(T\acute{N})).$
\end{theorem}

\begin{proof}
By using (\ref{int-geo-2}), (\ref{21}), (\ref{21a}) and (\ref{23}), we find%
\begin{equation*}
0=g(\breve{\nabla}_{W}\breve{J}U,\breve{J}N)-pg(\breve{\nabla}_{W}U,\breve{J}%
N),
\end{equation*}%
where $W,U\in \Gamma (S(T\acute{N}))$ and $N\in \Gamma (ltr(T\acute{N})).$ So,
we have
\begin{equation*}
0=g(h^{s}(W,\breve{J}U)-ph^{s}(W,U),\breve{J}N).
\end{equation*}%
Hence, we get the conclusion.
\end{proof}

\begin{theorem}
Let $\acute{N}$ be a radical screen transversal lightlike submanifold of a
locally metallic semi-Riemannian manifold $(\breve{N},\breve{g},\breve{J})$.
Then the radical distribution defines a totally geodesic foliation if and
only if one of the followings hold:

\begin{description}
\item[(i)] $A_{\breve{J}U}W$ belongs to $\Gamma (Rad(T\acute{N}))$ and $%
A_{U}^{\ast }W=0 $,

\item[(ii)] $A_{\breve{J}U}W=pA_{U}^{\ast }W$,

\item[(iii)] there is no component of $h^{s}(W,\breve{J}Z)$ in $\breve{J}%
Rad(T\acute{N}),$
\end{description}
\end{theorem}

for any $W,U\in \Gamma (Rad(T\acute{N}))$ and $Z\in \Gamma (S(T\acute{N})).$

\begin{proof}
From (\ref{int-geo-1}), for $W,U\in \Gamma (Rad(T\acute{N}))$ and $Z\in \Gamma
(S(T\acute{N})),$ we have%
\begin{eqnarray*}
0 &=&\breve{g}(\breve{\nabla}_{W}\breve{J}U,\breve{J}Z)-p\breve{g}(\breve{%
\nabla}_{W}U,\breve{J}Z) \\
&=&\breve{g}(-A_{\breve{J}U}W+pA_{U}^{\ast }W,\breve{J}Z),
\end{eqnarray*}%
by virtue of (\ref{21}), (\ref{21a}) and (\ref{23}), which implies either
(i) or (ii). Also, we can write%
\begin{equation*}
\breve{g}( h^{s} ( W,\breve{J}Z) ,\breve{J}U) =\breve{g}%
( A_{\breve{J}U}W,\breve{J}Z) =0,
\end{equation*}%
by virtue of%
\begin{equation*}
\breve{g}( A_{\breve{J}U}W,\breve{J}Z) =p\breve{g}(\breve{\nabla}%
_{W}U,\breve{J}Z)=0.
\end{equation*}

Therefore, the proof is completed.
\end{proof}

\begin{theorem}
Let $\acute{N}$ be a radical screen transversal lightlike submanifold of a
locally metallic semi-Riemannian manifold $(\breve{N},\breve{g},\breve{J})$.
Then the induced connection on $\acute{N}$ is a metric connection if and
only if either there is no component of $A_{\breve{J}\xi }W$ in $S(T\acute{N}%
)$ or there is no component of $h^{s}(U,W)$ in $\breve{J}Rad(T\acute{N}),$
for any $W,U\in \Gamma (S(T\acute{N})),$ $\xi \in \Gamma (Rad(T\acute{N})).$
\end{theorem}

\begin{proof}
Since $(\breve{N},\breve{g},\breve{J})$ is a locally metallic
semi-Riemannian manifold, then, for $W,U\in \Gamma (S(T\acute{N}))$ and $\xi
\in \Gamma (Rad(T\acute{N})),$ we have
\begin{equation*}
g(\breve{\nabla}_{W}\breve{J}\xi ,U)=g(\breve{\nabla}_{W}\xi ,\breve{J}U).
\end{equation*}%
By using (\ref{11}), (\ref{21}) and (\ref{21a}), we find%
\begin{equation*}
-g(A_{\breve{J}\xi }W,U)=g(\nabla _{W}\xi ,\breve{J}U),
\end{equation*}%
which implies that, either there is no component of $A_{\breve{J}\xi }W$ in $%
S(T\acute{N})$ or from (\ref{12}) in last equation, we have
\begin{equation*}
-g(h^{s}(U,W),\breve{J}\xi )=g(\nabla _{W}\xi ,\breve{J}U).
\end{equation*}%
So, we complete the proof.
\end{proof}

\subsection{Isotropic Screen Transversal Lightlike Submanifolds}

In case when $\acute{N}$ is an isotropic screen transversal lightlike
submanifold of a metallic semi-Riemannian manifold $(\breve{N},\breve{g},%
\breve{J})$, from the Definition \ref{def1} and Proposition \ref{prop2}, we
can write%
\begin{equation*}
T\acute{N}=Rad(T\acute{N})
\end{equation*}%
and tangent bundle of the main space has the decomposition%
\begin{equation*}
T\breve{N}=\{T\acute{N}\oplus ltr(T\acute{N})\}\oplus_{ort} \{\breve{J}Rad(T%
\acute{N})\oplus \breve{J}ltr(T\acute{N})\}\oplus_{ort} D_{o}.
\end{equation*}

\begin{theorem}
Let $\acute{N}$ be an isotropic screen transversal lightlike submanifold of
a locally metallic semi-Riemannian manifold $(\breve{N},\breve{g},\breve{J})$%
. In this case, $\acute{N}$ is totally geodesic if and only if $D^{l}(\xi
_{1},\breve{J}\xi _{2})=0$ and $D^{l}(\xi _{1},Z)=0$ and there is no
component of $D^{s}(\xi _{1},N)$ in $\Gamma (\breve{J}Rad(T\acute{N})),$ for
any $\xi _{1},$ $\xi _{2}$ $\in $ $\Gamma (Rad(T\acute{N})),$ $N\in \Gamma
(ltr(T\acute{N}))$ and $Z\in \Gamma (D_{0}).$
\end{theorem}

\begin{proof}
From  (\ref{21}) and (\ref{21a}), we find
\begin{equation*}
\breve{g}(\breve{\nabla}_{\xi _{1}}\breve{J}\xi _{2},\xi )=\breve{g}(\breve{%
\nabla}_{\xi _{1}}\xi _{2},\breve{J}\xi ).
\end{equation*}%
Then, for $\breve{J}\xi _{2}$ $\in $ $\Gamma (\breve{J}Rad(T\acute{N}))\subset
S(T\acute{N}^{\perp }),$ from (\ref{11}), we obtain%
\begin{equation}
\breve{g}(D^{l}(\xi _{1},\breve{J}\xi _{2}),\xi )=\breve{g}(h^{s}(\xi
_{1},\xi _{2}),\breve{J}\xi ).  \label{*}
\end{equation}%
Similarly, we have%
\begin{equation*}
\breve{g}(\breve{\nabla}_{\xi _{1}}\breve{J}\xi _{2},N)=\breve{g}(\breve{%
\nabla}_{\xi _{1}}\xi _{2},\breve{J}N).
\end{equation*}%
Using (\ref{9}) and (\ref{11}), we get
\begin{equation}
-\breve{g}(D^{s}\left( \xi ,N\right) ,\breve{J}\xi _{2})=\breve{g}(h^{s}(\xi
_{1},\xi _{2}),\breve{J}N).  \label{**}
\end{equation}%
Also, since $\breve{\nabla}$ is a metric connection, for $Z\in \Gamma
(D_{0}),$ from (\ref{9}) and (\ref{11}) again, we have%
\begin{eqnarray}
\breve{g}(\breve{\nabla}_{\xi _{1}}\xi _{2},Z) &=&-\breve{g}(\xi _{2},\breve{%
\nabla}_{\xi _{1}}Z)  \notag \\
\breve{g}\left( h^{s}(\xi _{1},\xi _{2}),Z\right)  &=&\breve{g}\left( \xi
_{2},D^{l}\left( \xi _{1},Z\right) \right) .  \label{***}
\end{eqnarray}%
(\ref{*}), (\ref{**}) and (\ref{***}) complete the proof.
\end{proof}

\bigskip

{\small {\ \textbf{Feyza Esra Erdo\u{g}an} \newline
Faculty of Education, Department of Elementary Education \newline
Ad\i yaman University, 02040, Ad\i yaman, Turkey \newline
e-mail: ferdogan@adiyaman.edu.tr }}

{\small \smallskip }

{\small \textbf{Selcen Y\"{u}ksel Perkta\c{s}} \newline
Department of Mathematics, Faculty of Arts and Sciences \newline
Ad\i yaman University, 02040, Ad\i yaman, Turkey \newline
e-mail: sperktas@adiyaman.edu.tr }

{\small \smallskip }

{\small \textbf{Bilal Eftal Acet} \newline
Department of Mathematics \newline
Ad\i yaman University, 02040, Ad\i yaman, Turkey \newline
e-mail: eacet@adiyaman.edu.tr }

{\small \smallskip }

{\small \textbf{Adara Monica Blaga} \newline
Faculty of Mathematics and Computer Science, Department of Mathematics 
\newline
West University of Timi\c{s}oara, 300223, Timi\c{s}oara, Rom\^{a}nia \newline
e-mail: adarablaga@yahoo.com  }

\end{document}